\pdfoutput=1
\documentclass{article}
\usepackage{multirow}%
\usepackage{amsmath,amssymb,amsfonts}%
\usepackage{amsthm}%
\usepackage{mathrsfs}%
\usepackage[title]{appendix}%
\usepackage{xcolor}%
\usepackage{textcomp}%
\usepackage{manyfoot}%
\usepackage{booktabs}%
\usepackage{algorithm}%
\usepackage{algorithmicx}%
\usepackage{algpseudocode}%
\usepackage{listings}%
\usepackage{lmodern}%
\usepackage{natbib}
\usepackage{hyperref}
\usepackage{geometry}
\usepackage{authblk}
\hypersetup{
    colorlinks=true,
    linkcolor=blue,
    filecolor=magenta,      
    urlcolor=cyan,
    pdftitle={Overleaf Example},
    pdfpagemode=FullScreen,
    }

\allowdisplaybreaks

\newcommand{\Hz}[1]{\mathcal{H}(#1)}

\newcommand{\pHz}[2]{\mathcal{H}^{#1}(#2)}

\newcommand{\HSt}[1]{\tilde{\gamma}_{#1}}
\newcommand{\HarmonicStieltjesconstants}[1]{\tilde{\gamma}_{#1}}
\newcommand{\zHSt}[1]{{}_2\tilde{\gamma}_{#1}}

\newcommand{\secondharmonicStieltjesconstants}[1]{{}_2\tilde{\gamma}_{#1}}
\newcommand{\thirdharmonicStieltjescosntants}[1]{{}_3\tilde{\gamma}_{#1}}
\newcommand{\fourthharmonicStieltjescosntants}[1]{{}_4\tilde{\gamma}_{#1}}
\newcommand{\anythharmonicStieltjescosntants}[2]{{}_{#1}\tilde{\gamma}_{#2}}

\newcommand{\powerharmonicetafunction}[2]{{\mathcal{J}}^{#1}(#2)}

\newcommand{\powerhl}[3]{\mathrm{Hl}_{#1}^{#2}(#3)}
\newcommand{\harmonicpolylogarithm}[2]{\mathrm{Hl}_{#1}\left(#2\right)}

\newtheorem{theorem}{Theorem}

\newtheorem{lemma}[theorem]{Lemma}%
\newtheorem{remark}{Remark}%
\newtheorem{corollary}[theorem]{Corollary}
\newtheorem{definition}{Definition}%

\author[1]{Lo Ho Tin}
\begin{document}

\title{On sums involving powers of harmonic numbers}
\maketitle

\abstract{In this paper, we study a Dirichlet series generated by powers of harmonic numbers. As an application of these functions, we derive certain series involving harmonic numbers. We also study the analytic properties of these Dirichlet series such as values negative integers and behavior at poles. In particular, objects similar to the Stieltjes constants are discussed. Asymptotics of the sums involving harmonic numbers are also studied. From these results I showed a connection between its analytic properties and a possible route to showing the irrationality of the Euler-Mascheroni constant.}

\section{Motivation}
The functions \begin{equation}\label{motivation}\sum_{n=1}^{\infty}\frac{H_n^{m}}{n^s}\quad ,\quad \sum_{n=1}^{\infty}\frac{H_n^{m}}{n^s}(-1)^{n+1}\end{equation}
where $H_k := 1 + 1/2 + \cdots + 1/k$ denotes the $k$th harmonic number, is not extensively studied for $m>1$ at the negative integers of $s$, Indeed there are study of closed form of $s$ at positive integers, e.g. \cite{bailey1994experimental}. In this paper, I will provide a comprehensive study of the functions in (\ref{motivation}). This paper will mostly focus on the negative integers of $s$ and positive integers $m$. 
\par
To demonstrate the application of the results in this paper, I will provide a solution to the following sums:
$$\sum_{n=1}^{\infty}\frac{H_n^m-(\log n+\gamma)^m}{n} \quad ,\quad \sum_{n=1}^{\infty}(-1)^{n+1}\left(H_n^m-(\log n+\gamma)^m\right)$$ That is, (\ref{Solved sum}) and (\ref{Alternation version of unsolved problem}). They will be appearing in the first few sections.
\par Let
$$\zeta(s) := \sum_{n=1}^\infty \frac{1}{n^s}, \quad (\Re(s) > 1) \quad ,\eta(s)=\sum_{n=1}^{\infty}\frac{(-1)^{n+1}}{n^s},\quad (\Re(s)>0)$$
denote the classical Riemann zeta function. The properties of the Riemann zeta function have been studied extensively in the past century due to its analytic properties. For example, the Riemann zeta function has been analytically continued to a meromorphic function in the complex plane with a simple pole at $s = 1$ and has the following Laurent series expansion:
\begin{equation}\label{Laurent of zeta}
    \zeta(s)=\frac{1}{s-1}+\gamma+\sum_{n=1}^{\infty} \frac{(-1)^n}{n!}\gamma_n (s-1)^n
\end{equation}
where $$\gamma_n := \lim_{m \to \infty} \left\{\sum_{k=1}^m \frac{\log^n k}{k} - \frac{\log^{n+1} m}{n+1}\right\}$$ defines the Stieltjes constants.

During correspondences with Goldbach, Euler studied sums of the form
$$\sum_{k=1}^\infty \frac{1}{k^n}\left(1 + \frac{1}{2^m} + \cdots + \frac{1}{k^m}\right), \quad (m \in \mathbb{N}, n \in \mathbb{N} \setminus\{1\}).$$
These types of sums are known as Euler sums in deference to Euler. Euler was able to derive a closed form in terms of zeta values for the case when $m=1$:
$$\sum_{k=1}^\infty \frac{H_k}{k^n} = \left(\frac{n}{2} + 1\right)\zeta(n+1) - \frac{1}{2}\sum_{k=1}^{n-2} \zeta(n-k)\zeta(k+1), \quad (n \in \mathbb{N}\setminus\{1\})$$
 \cite{apostol1984dirichlet} later studied the analytic properties of the following Dirichlet series:
$$\mathcal{H}(s) = \sum_{n=1}^\infty \frac{H_n}{n^s}, \quad (\Re(s) > 1).$$
This was also studied by \cite{matsuoka1982values} with a similar technique, namely, utilizing the Euler Macluarin summation formula. The function $\mathcal{H}$ defines the harmonic zeta function. As its name implies, it shares similarities to the Riemann zeta function. In this paper, we will consider a generalization of this function
 The following functions will be our main objects of study.
\begin{definition}
    Define the functions 
    \[ \pHz{m}{s}:=\sum_{n=1}^{\infty} \frac{H^m_n}{n^s}\quad, \quad 
          \mathcal{J}(s)=\sum_{n=1}^{\infty} \frac{(-1)^{n+1}}{n^s}H_n \quad ,\quad\powerharmonicetafunction{m}{s}=\sum_{n=1}^{\infty} \frac{(-1)^{n+1}}{n^s}H_n^m 
    \]
\end{definition} The notation omits $m$ when $m=1$. \par
Each section will be briefly summarized below. Section 2-4 develops the properties of $\mathcal{H}^m$ and applies it to solve (\ref{Solved sum}). Section $5$ develops alternating sums $\mathcal{J}^m$, finds their values at $0$, and use it so solve (\ref{Alternation version of unsolved problem}). The later sections are development motivated by the mentioned unsolved problems. Section $6$ provides a way to find the values of $\mathcal{J}$ and $\mathcal{J}^2$ at negative integers that could be represented explicitly, and states that $\mathcal{J}^m$ at negative integers are able to be written in terms of values of $\mathcal{J}^a$ where $a<m$, up to a certain number of the argument. Section $7$ derives a $q$ analog and shows that such investigation of $q$-analog could be linked to number theoretic studies. In Section $8$, we show a stunning relation among Dirichlet series and use it to develop a relation between $\mathrm{Hl}_{-k}^m(x)$ and $\mathrm{Hl}_{-k}^a(x)$ where $a<m$ (See definition \ref{harmonic polylogarithm}). In section $9$, asymptotic behaviors of harmonic sums are given. Section 10 studies the poles function $\mathcal{H}^m$ and find a relation among its coefficients.

To prove the later results we may use the asymptotic representation of $H_n$ shown below, as shown in \cite{apostol1984dirichlet}
\begin{equation}\label{Asymp. Harmonic}H_n=\log(n)+\gamma +\frac{1}{2n}-\sum_{a=1}^k \frac{B_{2a}}{2an^{2a}} +\int_n^{\infty} \frac{\tilde B_{2k}(x)}{x^{2k+1}}\, dx\end{equation}
where $B_n$ are the Bernoulli numbers and $\tilde B_n(x)$ are the periodic Bernoulli polynomials \cite{lehmer1988new}.

One could use it to derive:
\begin{equation}\label{EM of Harmonic zeta function v1}
    \Hz{s}=-\zeta'(s)+\gamma \zeta(s)+\frac{1}{2}\zeta({s+1})-\sum_{a=1}^k \frac{B_{2a}}{2a} \zeta({2a+s}) +\sum_{n\geq 1}\frac{1}{n^s}\int_n^{\infty} \frac{\tilde B_{2k}(x)}{x^{2k+1}}\, dx
\end{equation}
It was studied from \cite{candelpergher2020laurent} that the function $\mathcal{H}(s)$ have a double pole at $s=1$. With that being known, its Laurent series was defined. The \textit{harmonic Stieltjes constants} $\tilde{\gamma}_k$ were defined as the following:
\begin{equation}\label{definition of harmonic Steiltjes}\mathcal{H}(s)=\frac{1}{(s-1)^2}+\frac{\gamma}{s-1}+\sum_{k=0}^{\infty} \frac{(-1)^k}{k!}\tilde{\gamma}_k (s-1)^k\end{equation}
The symbol $\tilde{\gamma}_k$ is of reference from the paper \cite{candelpergher2020laurent}.

The positive values of $\mathcal{H}^m(s)$ were studied from \cite{flajolet1998euler} and \cite{bailey1994experimental}, and the positive values of $\mathcal{J}^m(s)$ are studied in \cite{flajolet1998euler}. In this paper, we pay attention to the divergent values of $\mathcal{H}^m(s)$ instead and infer results like in the coefficients of the Laurent expansion (\ref{Laurent of zeta}). More of the behavior of poles are discussed in section 10.
\section{A particular case}

To give an example of the main idea of the following sections, we consider the case where $m=2$, consider (\ref{Asymp. Harmonic}). Multiplying by $\frac{H_n}{n^s}$ and summing both sides gives
\begin{lemma} for $\Re (s) >1- k$
 
\begin{equation}\label{EM of H2 v1}
    \pHz{2}{s}=-\mathcal{H}'(s)+\gamma \mathcal{H}(s)+\frac{1}{2}\Hz{s+1}-\sum_{a=1}^k \frac{B_{2a}}{2a} \Hz{2a+s} +\sum_{n\geq 1}\frac{H_n}{n^s}\int_n^{\infty} \frac{\tilde B_{2k}(x)}{x^{2k+1}}\, dx
\end{equation}
   
\end{lemma}
We can see that the poles at $s=1$ comes from $-\mathcal{H}'(s)$ and $\gamma \mathcal{H}(s)$. Consider expressing $\mathcal H$ as in $(\ref{definition of harmonic Steiltjes})$, by differentiating both sides and putting the expression to the above, we realize that:
\begin{equation}\label{Lau. expansion of H2}
    \pHz{2}{s}=\frac{2}{(s-1)^3} +\frac{2\gamma}{(s-1)^2}+\frac{\gamma^2}{s-1}+O(1)
\end{equation}

We can hereby define the coefficients of $\pHz{2}{s}$:

\begin{definition}
    Define the constants $\zHSt{n}$ as the following
    \[\pHz{2}{s}=\frac{2}{(s-1)^3} +\frac{2\gamma}{(s-1)^2}+\frac{\gamma^2}{s-1}+\sum_{n=0}^{\infty} \frac{(-1)^n}{n!}\zHSt{n}(s-1)^n\]
\end{definition}

\begin{theorem} The constants $\zHSt{n}$ has the following representation
    \begin{equation}\label{zhst rep. v1}
   \zHSt{n}=\lim_{s\to 0^-}\frac{d^n}{ds^n}\left\{\pHz{2}{1-s}+\frac{2}{s^3}-\frac{2\gamma}{s^2}+\frac{\gamma^2}{s}\right\} \end{equation}
\end{theorem}
\begin{proof}
    Set $s\to 1-s$ and use the coefficients of its Taylor series.
\end{proof}
    For $\Re(s)<0$, we have the following
    \[\int_1^{\infty} \frac{\log^2(t)}{t^{1-s}}\, dt=-\frac{2}{s^3}\quad , \quad \int_1^{\infty}\frac{\log(t)}{t^{1-s}}\, dt=\frac{1}{s^2} \quad ,\quad \int_1^{\infty}\frac{1}{t^{1-s}}\, dt=-\frac{1}{s}\]

Putting them into (\ref{zhst rep. v1}) respectively gives

 \begin{theorem}
     \begin{equation}\label{zhst rep. v2}
         \zHSt{n}=\lim_{N\to \infty}\left\{\sum_{k= 1}^N \frac{H_k^2}{k} \log^n(k)-\frac{\log^{n+3}(N)}{n+3}-2\gamma \frac{\log^{n+2}(N)}{n+2}-\gamma^2 \frac{\log^{n+1}(N)}{n+1}\right\}
     \end{equation}
 \end{theorem}

Note that the condition $\Re (s)<0$ is necessary since we are taking the limit of $s$ to $0$ from the negative side.
 
 We can use (\ref{zhst rep. v2}) to solve the following problems:

 \begin{equation}\label{Problem from Vincent}
         \sum_{n\geq 1} \frac{(H_n-\log n)^3-\gamma^3}{n}=\frac{43 \pi^4}{720}-\frac{3}{4}\gamma^4-\gamma_3 -3\zHSt{1} +3\HSt{2}
    \end{equation}
    
\begin{equation}\label{Harmonic summ} \sum_{n\geq 1}\frac{H_n}{n}\left(H_n-\log(n)-\gamma\right)=\frac{5}{3}\zeta(3)-\frac{\gamma}{2}\zeta(2)-\frac{\gamma^3}{6}-\HSt{1}\end{equation}
I will omit the proofs here. Later on, we will develop a more general form of similar sums which we can easily reduce them to the above.

If we iterate the process of transforming $\pHz{m}{s}$ to $\pHz{m+1}{s}$, we can get a general picture of their poles.

\begin{theorem}[Distribution of poles]\label{Distribution of poles of H}
    The order of poles of $\pHz{m}{z}$ for $m\geq 1$ is distributed in the form
    \begin{equation}
        z\in (1,0,-1,-2,-3,-4,-5,-6, \ldots ) \longrightarrow (m+1,m,m,m-1,m,m-1,m,m-1,\ldots )
    \end{equation}
\end{theorem}

So a $m+1$ th order pole at $z=1$, $m$ th order pole at $z=0$, $m$ th order pole at $z=-1$, $m-1$ th order pole at $z=-2$ and so on. For example,
the poles of $\pHz{2}{s}$ are distributed as
\begin{itemize}
    \item Triple pole at $s=1$
    \item Double pole at $s=0,-1,-3,-5, \ldots$
    \item Simple poles at $s=-2,-4,-6 , \ldots$
\end{itemize}
Interestingly, If we take $m=0$ and aver that the above statement is true we get that $\mathcal{H}^{0}(z)$ has $-1$ order poles at negative even integers, which could be interpreted as a single zero. This corresponds to the trivial zeros of the Riemann zeta function.

\section{General case}

Applying the same technique in (\ref{Lau. expansion of H2}) to find the Laurent expansions of $\mathcal{H}^3, \mathcal{H}^4 ,\cdots$ , one may immediately recognize a pattern, stated it formally in the below. The proof is done easily by induction.

\begin{lemma} \label{ The principal part of the Laurant expansion of Hm} The principal part of the Laurant expansion of $\pHz{m}{s}$ is 
    \begin{equation}
        \mathcal{H}^m(s)=\sum_{j=0}^m \binom{m}{j}\frac{\gamma^{m-j}j!}{(s-1)^{j+1}}+O(1)
    \end{equation}
\end{lemma}
Knowing such expansion, we can define the following coefficients.
\begin{definition} Define the constants $\anythharmonicStieltjescosntants{m}{n}$ as the following
\begin{equation}
    \pHz{m}{s}=\sum_{j=0}^m \binom{m}{j}\frac{\gamma^{m-j}j!}{(s-1)^{j+1}}+\sum_{n=0}^{\infty}\frac{(-1)^n}{n!}\anythharmonicStieltjescosntants{m}{n}(s-1)^n
\end{equation}

\end{definition}
\begin{theorem} For all $m\geq 0$, $n \geq 0$:
    \begin{equation}\label{general limit representation of m harmonicStieltjes n}
        \anythharmonicStieltjescosntants{m}{n}=\lim_{N\to \infty}\left[\sum_{k=1}^N \frac{H^m_k \log^n(k)}{k}-\sum_{j=0}^{m} \binom{m}{j}\gamma^{m-j}\frac{\log^{n+j+1}(N)}{n+j+1}\right]
    \end{equation}
\end{theorem}
\begin{proof}
    The procedure is the same as in (\ref{zhst rep. v2}) using lemma \ref{ The principal part of the Laurant expansion of Hm}.
    
    \textit{Note that for} $\Re(s)>1$: \begin{equation}
    \int_1^{\infty}\frac{\log^j(t)}{t^s}\, dt =\frac{j!}{(s-1)^{j+1}}
\end{equation}

\end{proof}
\section{Constant terms}

In this section, we will find a closed form of the constant terms (i.e. $n=0$) in terms of Euler sums. We will proceed by studying the sum $\displaystyle S_m=\sum_{n=1}^N \frac{H_n^m}{n}$ at $N=\infty$. Since the author has already answered the proposed question on math stack exchange, some of the following section could also be seen in these \href{https://math.stackexchange.com/questions/2582497/a-closed-form-of-the-family-of-series-sum-k-1-infty-frac-lefth-k-rig/4956908#4956908}{answers} \cite{4956908},\cite{4972514}. For completeness the derivation will be given in the following sections.

Let $(a_n)$ and $(b_n)$ be two sequences. We will use the little o notation. That is, $a_n=b_n+o(1)$ if $\lim_{n\to\infty}(a_n-b_n)=0$. We start to tackle $S_m$ with summation by parts, taking the form $H_{n+1}^m-H_n^m=\sum_{k=0}^{m-1} \binom{m}{k} (-1)^{m+k+1}\frac{H^k_{n+1}}{(n+1)^{m-k}}$ gives
\begin{align*}S_m=& H_{N+1}^{m}H_N-\sum_{n=1}^N H_n (H_{n+1}^m-H_n^m) \\ 
=& H_{N+1}^{m+1}-\frac{H^m_{N+1}}{N+1}+\sum_{k=0}^{m-1}\binom{m}{k}(-1)^{m-k}\left\{\sum_{n=1}^{N+1} \frac{H_n^{k+1}}{n^{m-k}}-\sum_{n=1}^{N+1} \frac{H_n^k}{n^{m-k+1}}\right\} \\ 
(m+1)S_m = & H_{N+1}^{m+1}+\sum_{k=0}^{m-2} \binom{m}{k}(-1)^{k+m}\pHz{k+1}{m-k}\\ &-\sum_{k=-1}^{m-2}\binom{m}{k+1}(-1)^{k+m+1}\pHz{k+1}{m-k}+o(1) \\
\end{align*}
  Simplifying the right hand side, we obtain the following expansion up to a constant term:
\begin{equation}\label{asymptotic of a harmonic sum}
  \sum_{n=1}^N \frac{H_n^{m }}{n}= \frac{H_{N+1}^{m+1}}{m+1}+(-1)^{m+1}\frac{\zeta(m+1)}{m+1}+\sum_{k=0}^{m-2} \binom{m+1}{k+1} \frac{(-1)^{k+m} \mathcal{H}^{k+1}{(m-k)}}{m+1}+o(1)
\end{equation}

One could easily verify that:
\begin{equation} \label{power of harmonic numbers asymp}\frac{H_{N+1}^{m+1}}{m+1} =\frac{\gamma^{m+1}}{m+1}+\sum_{j=0}^m \binom{m}{j} \frac{\log^{j+1}(N)}{j+1}\gamma^{m-j}+o(1)\end{equation}
Putting this and $(\ref{asymptotic of a harmonic sum})$ back into (\ref{general limit representation of m harmonicStieltjes n}), the sums cancel out and we obtain:
\begin{theorem}[Closed form of constant terms] For all natural numbers $m$, the constant terms have the following closed form:
    \begin{equation}
    \anythharmonicStieltjescosntants{m}{0}=\frac{\gamma^{m+1}}{m+1}+(-1)^{m+1}\frac{\zeta(m+1)}{m+1}+\sum_{k=0}^{m-2} \binom{m+1}{k+1} \frac{(-1)^{k+m} \pHz{k+1}{m-k}}{m+1}
\end{equation}
\end{theorem}
\textit{Examples:}
\[\HarmonicStieltjesconstants{0}=\frac{\gamma^2}{2}+\frac{\zeta(2)}{2}\]
\[\secondharmonicStieltjesconstants{0}=\frac{\gamma^3}{3}+ \frac{5}{3}\zeta(3)\]
\[ \thirdharmonicStieltjescosntants{0}=\frac{\gamma^4}{4}+\frac{43}{8}\zeta(4)\]
\[\fourthharmonicStieltjescosntants{0}=\frac{\gamma^5}{5}+\frac{79}{5}\zeta(5)+3\zeta(2)\zeta(3)\]

Having such a tool, we can now easily tackle sums like:
\begin{theorem}
\begin{equation}\label{Solved sum}\begin{split}\sum_{n=1}^{\infty}\frac{H_n^m-(\log n+\gamma)^m}{n}= &{}_m \tilde\gamma_{0}-\sum_{k=0}^{m}\binom{m}{k}\gamma^{m-k}\gamma_k \\ =& -\frac{m}{m+1}\gamma^{m+1}+(-1)^{m+1}\frac{\zeta(m+1)}{m+1}\\ & +\sum_{k=0}^{m-2}\binom{m+1}{k+1}\frac{(-1)^{k+m}\mathcal{H}^{k+1}(m-k)}{m+1} 
-\sum_{k=1}^{m}\binom{m}{k}\gamma_k\gamma^{m-k}\end{split}\end{equation}
\end{theorem}

\begin{proof} [\proofname\ 1]
    The sum converges, considering the contant terms of the Laurent expansion of $\sum_{n=1}^{\infty}\frac{H_n^m-(\log n+ \gamma)^m}{n^{s}}$ around $s=1$ proves the theorem. 
\end{proof}

\begin{proof} [\proofname\ 2] The following is a proof without using the analytical tools developed in this paper. We can write the infinite sum in (\ref{Solved sum}) as the limit of finite sums.
    \begin{equation*}\begin{split}&
\lim_{N\to\infty}\left[\sum_{n=1}^N\frac{H^m_n}{n}-\sum_{n=1}^N\frac{(\log n+\gamma)^m}{n}\right]\\ =& \lim_{N\to\infty}\left[\sum_{n=1}^N\frac{H^m_n}{n}-\sum_{n=1}^N\frac{1}{n}\sum_{k=0}^m\binom{m}{k}\gamma^{m-k}\log^k(n)\right] \\
= & \lim_{N\to\infty}\Bigg[\sum_{n=1}^N\frac{H^m_n}{n}-\sum_{k=0}^m\binom{m}{k}\gamma^{m-k}{\underbrace{\sum_{n=1}^N\frac{1}{n}\log^k(n)}_{\gamma_k+\frac{\log^{k+1}(N)}{k+1}}}\Bigg]
\end{split}\end{equation*}
Where using (\ref{asymptotic of a harmonic sum}) and (\ref{power of harmonic numbers asymp}) finishes the proof.
\end{proof}

\textit{Examples:}

$$ \sum_{n=1}^{\infty}\frac{H_n-\log n-\gamma}{n}=\frac{\zeta(2)}{2}-\frac{\gamma^2}{2}-\gamma_1$$
$$ \sum_{n=1}^{\infty}\frac{H_n^2-(\log n+\gamma)^2}{n}=\frac{5\zeta(3)}{3}-\frac{2\gamma^3}{3}-2\gamma\gamma_1-\gamma_2$$
$$ \sum_{n=1}^{\infty}\frac{H_n^3-(\log n+\gamma)^3}{n}=-\frac{3}{4}\gamma^4+\frac{43}{8}\zeta(4)-3\gamma^2\gamma_1-3\gamma\gamma_2-\gamma_3$$
$$ \sum_{n=1}^{\infty}\frac{H_n^4-(\log n+\gamma)^4}{n}=-\frac{4}{5}\gamma^5+\frac{79}{5}\zeta(5)+3\zeta(2)\zeta(3)-4\gamma^3\gamma_1-6\gamma^2\gamma_2-4\gamma\gamma_3-\gamma_4$$

Similar formulas could be obtained, like the following:

\begin{equation}
    \sum_{n=1}^{\infty}\frac{(H_n-\log n)^m -\gamma^m}{n}=-\gamma^{m+1}+\sum_{k=0}^{m}(-1)^k\binom{m}{k}{}_{m-k}\tilde\gamma_k
\end{equation}
However, the author has not yet found a closed form of ${}_m\tilde \gamma_k $ for $k\geq 1$.
\begin{equation}\begin{split}
    \sum_{n=1}^{\infty}\frac{H^m_n(H_n-\log n-\gamma)}{n}=& {}_m\tilde\gamma_1 -\frac{\gamma^{m+2}}{(m+1)(m+2)}+(-1)^m\frac{\zeta(m+2)}{m+2}+(-1)^m \gamma \frac{\zeta(m+1)}{m+1}\\ & +\sum_{k=0}^{m-1}\binom{m+2}{k+1}\frac{(-1)^{k+m+1}}{m+2}\mathcal{H}^{k+1}(m+1-k)\\ & -\gamma \sum_{k=0}^{m-2}\binom{m+1}{k+1} \frac{(-1)^{k+m}}{m+1}\pHz{k+1}{m-k}
\end{split}\end{equation}

The above are general cases for (\ref{Problem from Vincent}) and (\ref{Harmonic summ}) respectively. 

\section{Alternating sums}
We briefly introduce and discuss the results inferred from studying the alternating sums $\mathcal{J}^m$. Consider the following result:
\begin{lemma}\label{coefficient to difference coefficient} If $a_0=0$, then
    \begin{equation} 
        \sum_{n=1}^{\infty} a_n x^n =\frac{1}{1-x}\sum_{n=1}^{\infty} (a_n-a_{n-1})x^n
    \end{equation}
\end{lemma}
The above lemma can be seen true by multiplying both sides by $1-x$
\begin{definition}\label{harmonic polylogarithm} Define the harmonic polylogarithm as
    \begin{equation}
        \powerhl{s}{m}{z}:=\sum_{n=1}^{\infty} \frac{H_n^m}{n^s}z^n 
 \quad ,\quad 
        \harmonicpolylogarithm{s}{z}:=\sum_{n=1}^{\infty} \frac{H_n}{n^s}z^n
    \end{equation}
\end{definition}
Indeed, this function shares similarity with the typical polylogarithm $\mathrm{Li}_s(z)=\sum_{n=1}^{\infty}\frac{z^n}{n^s}$.
\begin{theorem}[Dirichlet's regularization of $\mathcal{J}^m$ at $0$]
    \begin{equation}\label{value of powerharmoniceta at 0}
        \powerharmonicetafunction{m}{0}=\frac{1}{2}\sum_{k=1}^{m} \binom{m}{k} (-1)^{k+1}\powerharmonicetafunction{m-k}{k}
    \end{equation} 
    \end{theorem}

    \textit{Examples:} \[\begin{split}
        \mathcal{J}(0) &= \frac{\log^2(2)}{2} \\
        \mathcal{J}^2(0)&= \frac{\zeta(2)}{4}-\frac{\log^2(2)}{2}\\
        \mathcal{J}^3(0)&= \frac{9}{16}\zeta(3)-\frac{3}{4}\zeta(2)\log(2)+\frac{\log^3(2)}{2}\\ 
        \mathcal{J}^4(0) &= 2\mathrm{Li}_4\left(\frac{1}{2}\right)-\frac{23}{16}\zeta(4)-\frac{1}{2}\zeta(3)\log(2)+\zeta(2)\log^2(2)-\frac{5}{12}\log^4(2)
    \end{split}\]
    \begin{proof}
We start by setting $a_n=H_n^m$ in lemma \ref{coefficient to difference coefficient} and take the form $H_n^m-H_{n-1}^m=\sum_{k=1}^m \binom{m}{k}\frac{(-1)^{k+1}}{n^k}H_n^{m-k}$. This gives

    \[\sum_{n=1}^{\infty} H_n^m x^n =\frac{1}{1-x}\sum_{n=1}^{\infty} x^n \sum_{k=1}^m \binom{m}{k}\frac{(-1)^{k+1}}{n^k}H_n^{m-k}\]
    \begin{equation}\label{harmonic power generating function in terms of lower weights}
\powerhl{0}{m}{x}=\sum_{k=1}^m \binom{m}{k} (-1)^{k+1}\frac{\powerhl{k}{m-k}{x}}{1-x}
    \end{equation}
    The left hand side is equal to $\lim_{s\to 0}\sum_{n=1}^{\infty} \frac{H_n^m}{n^s} x^n$.  Which if we take $x\to -1^{+}$ both sides we get
    \[\lim_{s\to 0}\powerhl{s}{m}{-1}=\sum_{k=1}^m \binom{m}{k} (-1)^{k+1}\frac{\powerhl{k}{m-k}{-1}}{2}\]
    Which proves the theorem.
    \end{proof}
    An application of (\ref{value of powerharmoniceta at 0}) is the following:
\begin{corollary}\label{Alternating harmonic power sum}
\begin{equation}\label{Alternation version of unsolved problem}\begin{split}
    &\sum_{n=1}^{\infty}(-1)^{n+1}\left(H_n^m-(\log n+\gamma)^m\right) \\ & =-\frac{\gamma^m}{2}+\sum_{k=1}^{m}\binom{m}{k}(-1)^{k+1}\left(\frac{\powerharmonicetafunction{m-k}{k}}{2}+\gamma^{m-k}\eta^{(k)}(0)\right)\end{split}
\end{equation}
\end{corollary}

\section{Particular values at negative integers} With a similar procedure from theorem \ref{Distribution of poles of H}, we can see that the function $\mathcal{J}^m(s)$ is holomorphic everywhere. We now find its values at negative integers.
Consider the expression, for $-1<y<1$ and $-\pi< \theta<\pi$:
\begin{equation}\label{S}\begin{split}
    S^m(y)=& \left(\mathrm{Hl}_{0}^m(y)-\frac{\theta^2}{2!}\mathrm{Hl}_{-2}^m(y)+\frac{\theta^4}{4!}\mathrm{Hl}_{-4}^m(y)-\cdots\right)\\ & +i\left(\frac{\theta}{1!}\mathrm{Hl}_{-1}^m(y)-\frac{\theta^3}{3!}\mathrm{Hl}_{-3}^m(y)+\frac{\theta^5}{5!}\mathrm{Hl}_{-5}^m(y)-\cdots\right)\end{split}
\end{equation}
It could be seen that this expression is equal to $\sum_{k=1}^{\infty}H_k^m y^k e^{i\theta k}$. Let's consider the $m=1$ case.
\begin{align*}
    &S^1(-y) = -\frac{\log(1+ye^{i\theta})}{1+ye^{i\theta}}\\
    &=  -\frac{1+y \cos \theta}{1+2y\cos \theta+y^2}\cdot \frac{1}{2}\log(1+2y\cos \theta+y^2)-\frac{y\sin \theta}{1+2y\cos\theta +y^2} \arctan \left(\frac{y\sin \theta}{1+y\cos \theta}\right) \\ & +i\left\{\frac{y\sin \theta}{1+2y\cos\theta +y^2}\cdot \frac{1}{2}\log(1+2y\cos \theta+y^2)- \frac{1+y \cos \theta}{1+2y\cos \theta+y^2} \arctan \left(\frac{y\sin \theta}{1+y\cos \theta}\right)\right\} \\ 
\end{align*}
Taking the limit as $y\to 1$ gives
$$\lim_{y\to 1}S^1(-y)= -\frac{1}{2}\log\left(2\cos\frac{\theta}{2}\right)-\frac{\theta}{4}\tan\left(\frac{\theta}{2}\right)+i\left\{-\frac{\theta}{4}+\frac{1}{2}\tan\left(\frac{\theta}{2}\right)\log\left(2\cos\frac{\theta}{2}\right)\right\}$$
If we expand the above as a Taylor series in $\theta$ and compare the coefficients with (\ref{S}). We get for $n\geq 1$:

\begin{align}\label{Negative integers of J}
    \mathcal{J}(0)=& \frac{1}{2}\log 2 \\
    \mathcal{J}(-1)= &\frac{1}{4}-\frac{1}{4}\log 2 \\
    \mathcal{J}(-2n)= &-\frac{B_{2n}}{4n}(4^n-1)(2n-1) \\
    \mathcal{J}(1-2n)= &-\frac{B_{2n}(4^n-1)}{2n}\log 2-\sum_{k=1}^{n-1}\frac{(2n-1)!B_{2k}B_{2n-2k}(4^k-1)(4^{n-k}-1)}{2k(2k)!(2n-2k)!}
\end{align}
A more general form can be seen in (\ref{harmonic polylogarithm in terms of polylog}).
\par
Note that the considered functions has their Taylor series $$\frac{1}{2}\tan\frac{\theta}{2}=\sum_{n=1}^{\infty}\frac{B_{2n}(-1)^{n+1}(4^n-1)}{(2n)!}\theta^{2n-1}=\frac{\theta}{1!}\eta(-1)-\frac{\theta^3}{3!}\eta(-3)+\frac{\theta^5}{5!}\eta(-5)-\cdots$$
$$\log\left(\cos \frac{\theta}{2}\right)=\sum_{n=1}^{\infty}\frac{B_{2n}(-1)^{n}(4^n-1)}{(2n)(2n)!}\theta^{2n}=-\frac{\theta^2}{2!}\eta(-1)+\frac{\theta^4}{4!}\eta(-3)-\frac{\theta^6}{6!}\eta(-5)+\cdots$$
With a slight modification of $S$, one could also prove that $\mathcal{H}(-2n)=\frac{B_{2n} (2n-1)}{4n}$ for $n\geq 1$. Note that the values of $\mathcal{H}(-2n)$ can also be found using the Euler Macluarin summation formula, notably \cite{matsuoka1982values} and \cite{apostol1984dirichlet}. The $m=0$ case gives the values of the Dirichlet eta function at negative values.

\cite{boyadzhiev2009alternating} obtained the above result using the Euler-Boole Summation formula. As seen, the above requires only elementary manipulations. Even better, this procedure works for $m=2$ with a reasonable amount of effort.
It should be anticipated that the coefficients soon get messy and complicated. Therefore, we will start writing the expansion as functions and stop writing their coefficients in its explicit closed form.
Using the identity $$\sum_{n=1}^{\infty}H_n^2 x^n=\frac{\log^2(1-x)+\mathrm{Li}_2(x)}{1-x}$$
we can deduce the following;
\begin{theorem}
   \begin{equation}\begin{split}
         \mathcal{J}^2(0)-&\frac{\theta^2}{2!}\mathcal{J}^2(-2)+\frac{\theta^4}{4!}\mathcal{J}^2(-4)-\cdots =\frac{1}{2}\left(\frac{\pi^2}{12}-\log^2\left(2\cos \frac{\theta}{2}\right)\right)\\ &-\frac{1}{2}\tan\frac{\theta}{2}\left( \theta \log\left(2\cos \frac{\theta}{2}\right)-\frac{\theta}{1!}\eta(1)+\frac{\theta^3}{3!}\eta(-1)-\frac{\theta^5}{5!}\eta(-3)+\cdots\right) \end{split}\end{equation}\begin{equation}\begin{split}
            \frac{\theta}{1!}\mathcal{J}^2(-1)-& \frac{\theta^3}{3!}\mathcal{J}^2(-3)+\frac{\theta^5}{5!}\mathcal{J}^2(-5)-\cdots = \frac{1}{2}\tan \frac{\theta}{2}\left(\log^2\left(2\cos \frac{\theta}{2}\right)-\frac{\pi^2}{12}\right)\\ &-\frac{1}{2}\left(\theta \log\left(2\cos \frac{\theta}{2}\right) -\frac{\theta}{1!}\eta(1)+\frac{\theta^3}{3!}\eta(-1)-\frac{\theta^5}{5!}\eta(-3)+\cdots\right)\end{split}
         \end{equation}
    These values can be written purely in terms of the $m=1$ and $m=0$ case: \begin{equation}\label{negative integers of J2}
        \mathcal{J}^2(-k)=\mathcal{J}(1-k)-\frac{\eta(2-k)}{2}+\sum_{j=0}^{k-1}\binom{k}{j}\eta(j-k)\left(\eta(2-j)-2\mathcal{J}(1-j)\right)
    \end{equation}
They are tabulated below:
       \begin{table}[h] \caption{Negative values of $\mathcal{J}^2$}\label{tab1}
        \begin{tabular}{|c|l|} \hline 
     $n$ & $\mathcal{J}^2(-n)$  \\ \hline 
     $0$  & $\frac{\pi^2}{24}-\frac{\log^2 2}{2}$  \\ \hline    
     $1$ & $ \frac{\log^2(2)}{4}-\frac{\pi^2}{48}$ \\ \hline  
     $2$ &  $-\frac{\log 2}{4}$ \\ \hline   
     $3$ & $\frac{\pi^2}{96}-\frac{1}{4}+\frac{3}{8}\log 2-\frac{\log^2 2}{8}$ \\ \hline  
     $4$ & $\frac{5}{16}+\frac{\log 2}{8} $ \\ \hline   
     $5$ & $\frac{23}{32}-\frac{\pi^2}{48}+\frac{\log^2 2}{4}-\frac{15}{16}\log 2$ \\ \hline  
     $6$ & $-\frac{49}{32}-\frac{\log2}{4}$  \\ \hline  
     $7$ & $-\frac{129}{32}+\frac{17 \pi^2}{192}-\frac{17}{16}\log^2 2+\frac{147}{32}\log 2$ \\ \hline  
     $8$ & $\frac{717}{64}+\frac{17}{16}\log 2 $ \\ \hline  
     $9$ & $\frac{4639}{128}-\frac{31}{48}\pi^2+\frac{31}{4}\log^2 2-\frac{1185}{32}\log 2$ \\ \hline  
     $10$& $-\frac{7711}{64}-\frac{31}{4}\log 2$ \\ \hline  
     $\cdots$ & $\cdots$
 \\ \hline\end{tabular}
    \end{table}
    
\end{theorem}
\begin{theorem}\label{negative values of Jm}
\textit{For positive values $n$, we can write $\mathcal{J}^m(-n)$ in terms of the values of $\mathcal{J}^{m-1},\mathcal{J}^{m-2},..,\eta$}. That is: \begin{equation}
    \begin{split}
        \mathcal{J}^{m+1}(-n)=& \{a_m(1)\mathcal{J}^m(1)\}+ \{a_{m-1}(2)\mathcal{J}^{m-1}(2)+a_{m-1}(1)\mathcal{J}^{m-1}(1)\}\\
         & \{a_{m-2}(3)\mathcal{J}^{m-2}(3)+a_{m-2}(2)\mathcal{J}^{m-2}(2)+a_{m-2}(1)\mathcal{J}^{m-2}(1)\}+\\ &\cdots+\{a_1(m)\mathcal{J}(m)+a_1(m-1)\mathcal{J}(m-1)+\cdots+a_1(1)\mathcal{J}(1)\}\\ &+\text{Values of }\eta
    \end{split}
\end{equation}
Where $a$'s are coefficients as rational numbers that depend on $n$. \end{theorem} 
In other words, we can write the negative values of $\mathcal{J}^m$ in terms of lower $m$, \textit{for example}:
\begin{equation}\label{Negative values of J3}\begin{split}
    \mathcal{J}^3(-k)=& \frac{1}{2}\left(3\mathcal{J}^2(1-k)-3\mathcal{J}(2-k)+\eta(3-k)\right)\\ &-\sum_{j=0}^{k-1}\binom{k}{j}\eta(j-k)\left\{3\mathcal{J}^2(1-j)-3\mathcal{J}(2-j)+\eta(3-j)\right\}\end{split}
\end{equation}
$a$'s being rational comes from the fact that the Bernoulli numbers are rational. Theorem \ref{negative values of Jm} is a particular case of the results in theorem \ref{Harmonic polylog theorem}.

The proof of (\ref{negative integers of J2}) and (\ref{Negative values of J3}) will be seen in the later section when we develop a more general result (See theorem \ref{Master theorem of Dirichlet series at negative integers}).

\section{A possible q analog}
In this section, we will find a $q$ analog of the harmonic zeta function. We will use the symbol $\phi_n$ to denote the infinite sum $$\frac{1^n xq}{1-q}+\frac{2^n x^2q^2}{1-q^2}+\frac{3^n x^3q^3}{1-q^3}+\frac{4^n x^4q^4}{1-q^4}+\cdots$$

Start by considering the expression:
$$S=\frac{xqe^{i\theta}}{1-q}+\frac{x^2q^2e^{2i\theta}}{1-q^2}+\frac{x^3q^3e^{3i\theta}}{1-q^3}+\frac{x^4q^4e^{4i\theta}}{1-q^4}+\cdots$$ Consider squaring $S$. 
If we expand as a Taylor series in $\theta$ first and square it via Cauchy product, we get \begin{align*}S^2& =\left(\phi_0+i\frac{\theta}{1!}\phi_1-\frac{\theta^2}{2!}\phi_2-i\frac{\theta^3}{3!}\phi_3+\frac{\theta^4}{4!}\phi_4+\cdots\right)^2\\ 
 &=\phi_0^2+i\frac{\theta}{1!}(\phi_0\phi_1+\phi_1\phi_0)-\frac{\theta^2}{2!}(\phi_0\phi_2+2\phi_1^2+\phi_2\phi_0)-\cdots\end{align*}

with the $\theta^n$ coefficient being $$\frac{i^n}{n!}\left\{\binom{n}{0}\phi_0\phi_{n}+\binom{n}{1}\phi_1\phi_{n-1}+\cdots+\binom{n}{n}\phi_{n}\phi_0\right\}$$
Meanwhile, if we square the series immediately we get
$$S^2=\sum_{n=2}^{\infty}x^n e^{i\theta n}\left(\frac{q}{1-q}\frac{q^{n-1}}{1-q^{n-1}}+\frac{q^2}{1-q^2}\frac{q^{n-2}}{1-q^{n-2}}+\cdots+\frac{q^{n-1}}{1-q^{n-1}}\frac{q}{1-q}\right)$$
Using the identity $$\frac{1}{(1-q^k)(1-q^{n-k})}=\frac{1}{1-q^n}\left(\frac{q^k}{1-q^k}+\frac{1}{1-q^{n-k}}\right)$$
simplifies the expression to
$$S^2=\sum_{n=2}^{\infty}\frac{x^ne^{i\theta n} q^n}{1-q^n}\left(n-1+2\left(\frac{q}{1-q}+\frac{q^2}{1-q^2}+\cdots+\frac{q^{n-1}}{1-q^{n-1}}\right)\right)$$
If we now expand the function in $\theta$ we get  $$S^2=\sum_{n\geq 0}\frac{\theta^ni^n}{n!}\left(\phi_{n+1}-\phi_n+2\sum_{k=1}^{\infty}\frac{k^nx^k q^k}{1-q^k}\left(\frac{q}{1-q}+\frac{q^2}{1-q^2}+\cdots+\frac{q^{k-1}}{1-q^{k-1}}\right)\right)$$
We can now compare coefficients of both sides to obtain \begin{equation}
    \sum_{k=2}^{\infty}\frac{k^nx^kq^k}{1-q^k}\left(\frac{q}{1-q}+\frac{q^2}{1-q^2}+\cdots+\frac{q^{k-1}}{1-q^{k-1}}\right)=\frac{1}{2}\left(\phi_n-\phi_{n+1}+\sum_{k=0}^n\binom{n}{k}\phi_k\phi_{n-k}\right)
\end{equation}
This is a $q$ analog to the harmonic zeta function. To see this, multiply both sides by $(1-q)^2$ and take the limit as $q$ to $1$ to obtain:
\begin{equation}\label{harmonic polylogarithm in terms of polylog}
    \mathrm{Hl}_{-n}(x)=\mathrm{Li}_{1-n}(x)+\sum_{k=0}^n\binom{n}{k}\mathrm{Li}_{k-n}(x)\mathrm{Li}_{1-k}(x)
\end{equation}
Taking the limit as $x$ going to $-1$ gives $\mathcal{J}$ at negative integers. We can also take the limit as $x\to i$ to study the functions $\sum_n \frac{(-1)^n H_{2n+1}}{(2n+1)^z}$ and $\sum_n \frac{H_{2n}}{n^z}$
\begin{remark}
   Note that, taking $x\to 1$ to find a closed form of $\mathcal H$ is not a valid procedure. Since $\mathrm{Li}_{-n}(x)$ is not analytic at $x=1$, taking $x\to 1$ on $\mathrm{Li}_{-n}(x)$ doesn't give the values of $\zeta(-n)$.
\end{remark}
Let us denote, for $\Theta_0=0$: $$\Theta_n=\frac{q}{1-q}+\frac{q^2}{1-q^2}+\cdots +\frac{q^n}{1-q^n}$$

\par
If we further consider the expressions:
$$T=\frac{1xq e^{i\theta}}{1-q}+\frac{2 x^2 q^2 e^{2i\theta}}{1-q^2}+\frac{3x^3 q^3e^{3i\theta}}{1-q^3}+\cdots$$
$$U=\frac{xq^2 e^{i\theta}}{(1-q)^2}+\frac{x^2 q^4 e^{2i\theta}}{(1-q^2)^2}+\frac{x^3 q^6e^{3i\theta}}{(1-q^3)^2}+\cdots$$

We can then show that, for $z=xe^{i\theta}$:
$$\sum_{n=1}^{\infty}\Theta_n ^2z^n=\frac{S^2+U-T+S}{1-z}$$
 which, by expanding in $\theta$ gives the $q$ analog of $\mathcal{J}^2(-n)$. i.e. \begin{equation}\label{Theta 2}
    \sum_{k=1}^{\infty}\Theta_{k}^2k^n x^k=2\kappa_n+\psi_n+\sum_{k=0}^n\binom{n}{k}\mathrm{Li}_{k-n}(x)\left(2\kappa_k+\psi_k\right)
\end{equation}
where
\begin{equation}
    \kappa_k=\sum_{n=2}^{\infty}\frac{n^kx^n q^n}{1-q^n}\Theta_{n-1} \quad ,\quad 
    \psi_k=\sum_{n=1}^{\infty}\frac{n^k x^n q^{2n}}{(1-q^n)^2}
\end{equation}
An alternative derivation could be seen in theorem \ref{Master theorem of Dirichlet series at negative integers}.
\par

This kind of study of $q$ extension has interesting number theoretic implications. If we consider $\Theta_n$,
 expanding this in $q$ gives $$\Theta_n=c_1(n)q+c_2(n)q^2+c_3(n)q^3+\cdots$$ Then $c_m(n)$ denotes the amount of divisors of $m$ that is less than or equal to $n$. For example $c_{12}(5)=4$ because $12$ has divisors $1,2,3,4,6,12$ and four of them are less than or equal to $5$.
 Denote $\sigma_0^{-}(n)=\sum_{d|n}(-1)^d$ and $\sigma_0(n)=\sum_{d|n} 1$, i.e. $\sigma_0(n)$ counts the number of divisors of $n$.
Then we can show that:
$$\sum_{n=1}^m(-1)^n c_m(n)=\begin{cases}\displaystyle
    \text{Counts the number of even divisors of $m$ when $m$ is even} \\
    \displaystyle \text{$\sigma_0(m)$, when $m$ odd}
\end{cases}$$
if we denote $\sigma_k^{-}(n)=\sum_{d|n}(-1)^d d^k $, then $$\sum_{n=1}^m(-1)^n n c_m(n)=\begin{cases}
    \displaystyle \frac{\sigma_1^-(m)}{2}-\frac{\sigma_0^-(m)}{4}+\frac{2m+1}{4}\sigma_0(m) \, ,\, m \text{ even}\\ 
    \displaystyle
    -\frac{\sigma_1(m)}{2}-\frac{m}{2}\sigma_0(m) \, ,\, m \text{ odd} 
\end{cases}$$
More generally, we can use theorem \ref{Master theorem of Dirichlet series at negative integers} to show that 
\begin{equation}\label{cm formula}\begin{split}\sum_{n=1}^m(-1)^n n^k&  c_m(n)= \sigma_k^-(m)+\eta(-k)\left\{\sigma_0(m)-\sigma_0^-(m)\right\}\\ &-\sigma_0(m)\left(1^k-2^k+3^k-\cdots+(-1)^{m+1}m^k\right)\\ &-\left\{\binom{k}{1}\eta(1-k)\sigma_1^-(m)+\binom{k}{2}\eta(2-k)\sigma_2^-(m)+\cdots+\binom{k}{k}\eta(0)\sigma_k^-(m)\right\}\end{split}\end{equation}
Where $\eta$ denotes the Dirichlet eta function. Note that when $m$ is odd, we can write (\ref{cm formula}) only in terms of the sum of divisors functions $\sigma_k(m)$. This is because $\sigma_k^{-}(m)=-\sigma_k(m)$ when $m$ is odd. We will not dive too deep into such results in this paper.

\section{The harmonic polylogarithm}
We will utilize a very general expression of the relation between two power series, and show that for positive integer $k$, the harmonic polylogarithm $\mathrm{Hl}_{-k}^m(x)$ can be written purely in terms of a combination of harmonic polylogarithms with lower $m$. \par
\begin{theorem}\label{Master theorem of Dirichlet series at negative integers}
 Let $$\begin{cases}\Phi_k(x)=a_1 1^kx+a_2 2^k x^2+a_3 3^kx^3+\cdots \\
 \varphi_k(x)=b_1 1^kx+b_2 2^k x^2+b_3 3^kx^3+\cdots\end{cases}$$
 Where $a_0=0$ and $b_n=a_n-a_{n-1}$. Then $\Phi$ and $\varphi$ has the relation
 $$\Phi_k(x)=\varphi_k(x)+\sum_{j=0}^k \binom{k}{j}\mathrm{Li}_{j-k}(x)\varphi_j(x)$$\end{theorem}
For example $a_n=H_n$ gives (\ref{harmonic polylogarithm in terms of polylog}), $a_n=H_n^2$ gives (\ref{negative integers of J2}), $a_n=H_n^3$ gives (\ref{Negative values of J3}) and $a_n=\Theta_n^2$ gives (\ref{Theta 2})
\begin{proof}Take lemma \ref{coefficient to difference coefficient}, set $x$ as $xe^{i\theta}$ and expand in $\theta$. \end{proof}
\begin{theorem}\label{Harmonic polylog theorem}
    For positive integers $m$, the infinite sum $$H_1^m 1^kx+H_2^m 2^kx^2+H_3^m 3^kx^3+H_4^m 4^kx^4+\cdots$$ can be purely written in terms of combinations of harmonic polylogarithms of lower $m$, governed by the equation: \begin{equation}\label{power harmonic polylog in terms of polylog}\mathrm{Hl}_{-k}^m(x)=\sum_{j=1}^m\binom{m}{j}(-1)^{j+1}\left(\mathrm{Hl}_{j-k}^{m-j}(x)+\sum_{l=0}^k\binom{k}{l}\mathrm{Li}_{l-k}(x)\mathrm{Hl}_{j-l}^{m-j}(x)\right)\end{equation}
\end{theorem}
\begin{proof}
Set $a_n=H_n^m$ gives the above formula. We see that for all positive integer $m$, we can write $\mathrm{Hl}^m_{-k}(x)$ in terms of $\mathrm{Hl}^{m-1}_{-k}(x), \mathrm{Hl}^{m-2}_{-k}(x),\cdots , \mathrm{Hl}_{-k}(x),\mathrm{Li}_{-k}(x)$. This means that we can write $\mathrm{Hl}^m_{-k}(x)$ merely in terms of polylogs $\mathrm{Li}_{-k}(x)$ and $\mathrm{Hl}^{a}_{k}(x)$ for positive $k$, $k+a\leq m$.\end{proof}
Though (\ref{power harmonic polylog in terms of polylog}) gives an explicit expression, is a very complicated form and I do not use this to compute anything in practice. It is more practical to put special cases in theorem \ref{Master theorem of Dirichlet series at negative integers} and write them down individually.

\begin{remark} (\ref{value of powerharmoniceta at 0}) is a particular case of the above.
\end{remark}
\section{Asymptotics}

Using (\ref{asymptotic of a harmonic sum}) we can deduce further asymptotic formulas for different kinds of sums. Again, the little o notation will be used. 

\begin{theorem}\label{Asymptotic of sums of alternating Harmonic powers} For all integers $m>0$:
\begin{equation}\begin{split}
    \sum_{n=1}^N(-1)^{n+1}H_n^m=&\frac{(-1)^{N+1}}{2}H_{N+1}^m+\sum_{k=0}^{m-1}\binom{m}{k}\frac{(-1)^{m-k-1}}{2}\powerharmonicetafunction{k}{m-k}+o(1)
    \end{split}
\end{equation}

\textit{Examples:}
\begin{align*}
    \sum_{n=1}^N(-1)^{n+1}H_n & =\frac{(-1)^{N+1}}{2}H_{N+1}+\frac{\ln(2)}{2}\\ 
    \sum_{n=1}^N(-1)^{n+1}H_n^2 & =\frac{(-1)^{N+1}}{2}H_{N+1}^2+\frac{\zeta(2)}{4}-\frac{\ln^2(2)}{2}\\  \sum_{n=1}^N(-1)^{n+1}H_n^3 & =\frac{(-1)^{N+1}}{2}H_{N+1}^3+\frac{9}{16}\zeta(3)-\frac{3}{4}\zeta(2)\ln(2)+\frac{\ln^3(2)}{2}
\end{align*}

\begin{proof}
    Starting with summation by parts, we get
    \[\begin{split}\sum_{n=1}^N(-1)^{n+1}H_n = &\frac{1+(-1)^{N+1}}{2}H_{N+1}^m-\frac{m}{2}\sum_{n=1}^{N+1}\frac{H_n^{m-1}}{n}-\sum_{k=0}^{m-2}\binom{m}{k}\frac{(-1)^{m+k+1}}{2}\pHz{k}{m-k} \\ &+ \sum_{k=0}^{m-1}\binom{m}{k}\frac{(-1)^{m+k+1}}{2}\powerharmonicetafunction{k}{m-k}+o(1)\end{split}\]
    Using (\ref{asymptotic of a harmonic sum}) proves our desired closed form
\end{proof}
\end{theorem}

Below we show the asymptotic behavior of $\sum_n H^m_n$.
\begin{theorem}\label{Asymptotic of sums of Harmonic powers}
    \begin{equation}\begin{split}
    \sum_{n=1}^{N}H_n& = (N+1)(H_{N+1}-1) \\ 
    \sum_{n=1}^{N}H_n^2& = (N+1)H^2_{N+1}-(2N+3)H_{N+1}+2N+2
    \end{split}\end{equation}
and for integers $M\geq 3$
\begin{equation}\label{asymptotic of sum of power harmonic number}\begin{split}
\sum_{n=1}^NH_n^M= & (N+1)H_{N+1}^M+(-1)^M M!\Bigg\{\sum_{m=3}^{M-1}\frac{(-1)^m}{m!}H_{N+1}^m\left(N+\frac{3}{2}\right)\\  &+N+1+\left(\frac{N}{2}+\frac{3}{4}\right)H_{N+1}^2- \left(N+\frac{3}{2}\right)H_{N+1}-\sum_{m=3}^M\frac{\zeta(m-1)}{m!}\left(\frac{m}{2}-1\right) \\ & +\sum_{m=4}^{M} \frac{1}{m!}\sum_{k=1}^{m-3}\left[\binom{m-1}{k}\frac{m}{2}-\binom{m}{k}\right](-1)^{k+1}\mathcal{H}^k (m-k-1)\Bigg\}+o(1)
\end{split}
\end{equation}

\textit{Examples:}
\begin{equation}\begin{split}
    \sum_{n=1}^N H_n^3 = &  (N+1) H_{N+1}^3-6N-6-\left(3N+\frac{9}{2}\right)H_{N+1}^2+ (6N+9)H_{N+1}+\frac{\zeta(2)}{2}+o(1) \\
    \sum_{n=1}^N H_n^4  = &  (N+1)H_{N+1}^4 -(4N+6)H_{N+1}^3 +24N+24  \\ & (12N+18)H_{N+1}^2 -(24N+36)H_{N+1}+3\zeta(3)-2\zeta(2)+o(1) \\ 
    \sum_{n=1}^N H_{n}^5 = & (N+1)H_{N+1}^5 -5\left(N+\frac{3}{2}\right)H_{N+1}^4+(20N+30)H_{N+1}^3-(60N+90)H_{N+1}^2 \\ &+(120N+180)H_{N+1}-120N-120+\frac{33}{2}\zeta(4)-15\zeta(3)+10\zeta(2)+o(1)
    \end{split}
\end{equation}
\end{theorem}
 The first two sums are stated for completeness, they can be derived via summation by parts, we focus on the general case.
\begin{proof}
    Using summation by parts and (\ref{asymptotic of a harmonic sum}), we can obtain the reduction formula
\begin{equation}\label{raw asymptotic power harmonic at 0}\begin{split}
    \sum_{n=1}^N H_n^m= & (N+1)H_{N+1}^m-\frac{m}{2}H_{N+1}^{m-1}-m\sum_{n=1}^N H_n^{m-1} + (-1)^{m-1}\zeta(m-1)\left(\frac{m}{2}-1\right) \\ &+\sum_{k=1}^{m-3}\left[\binom{m-1}{k}\frac{m}{2}-\binom{m}{k}\right](-1)^{m-k-1}\mathcal{H}^k(m-k-1)+o(1)
\end{split}\end{equation}
for $m\geq 3$. Multiply both sides by $\frac{(-1)^m}{m!}$ and summing both sides from $m=3$ to $M$ gives the theorem.
\end{proof}
One can see that (\ref{asymptotic of sum of power harmonic number}) is a very complicated form. To see the terms more clearly, note that we have $(N+1)H^M_{N+1}$ being the dominant term, $(N+\frac{3}{2})H^{M-1}_{N+1},\cdots,(N+\frac{3}{2})H^{3}_{N+1},(\frac{N}{2}+\frac{3}{4})H_{N+1}^2,-(N+\frac{3}{2})H_{N+1},N+1$ and the constant terms being the two summations, all multiplied by $(-1)^M M!$.

\section{Behavior at poles}
In this section, we will first study the residues of $\mathcal{H}^m$, and use the results to study higher-order poles.
I will use $H^m(n)$ to denote the residue of $\mathcal{H}^m(z)$ at the point $z=n$. We know that from theorem \ref{Distribution of poles of H} that $\mathcal{H}^m$ has its poles at $1,0,-1,-2,-3,\cdots$. If we were to list out all the residues using (\ref{Asymp. Harmonic}), we get:

\begin{equation}\label{Residues}\begin{split}
       H^{m+1}(n)& = \gamma H^m(n)+\frac{1}{2}H^m(n+1)+\sum_{k=1}^N\zeta(1-2k)H^{m}(2k+n) \\
       H^{m+1}(1) & =\gamma H^{m}(1) \\
       H^{m+1}(0) &=\gamma H^m(0)+\frac{1}{2} H^m(1)\\
       H^{m+1}(-1) & =\gamma H^m(-1)+\frac{1}{2}H^m(0)+\zeta(-1)H^m(-1)\\
       H^{m+1}(-2) & = \gamma H^m(-2)+\frac{1}{2} H^m(-1) +\zeta(-1) H^m (0) \\ 
       H^{m+1}(-3) & = \gamma H^m (-3)+\frac{1}{2}H^m (-2)+\zeta(-1)H^m(-1)+\zeta(-3)H^m(1) \\
       H^{m+1}(-4) & = \gamma H^m (-4)+\frac{1}{2}H^m (-3)+\zeta(-1)H^m(-2)+\zeta(-3)H^m(0) \\
       H^{m+1}(-5) & = \gamma H^m (-5)+\frac{1}{2}H^m (-4)+\zeta(-1)H^m(-3)+\zeta(-3)H^m(-1)+\zeta(-5)H^m(1) \\
       \vdots
\end{split}\end{equation}
where $N$ is a suitably chosen constant depending on the point of residue. Note the fact that $\zeta(1-n)=(-1)^{n-1}\frac{B_n}{n}$.
\par
Upon listing out the residues, we find out that we can write $H^m(1-k)$ in the form \begin{equation}\label{an}
   \gamma^m a_k(0)+\sum_{n=1}^k \binom{m}{n}\gamma^{m-n}a_k(n)
\end{equation}

We can deduce the reduction formula for $a_n$ if we solve the system of reduction formulas in (\ref{Residues}) one by one. Its reduction formula is:

\begin{equation}\label{reduction of an}
    a_{n}(\ell +1)=\sum_{k=1}^{n-\ell} (-1)^k \zeta(1-k)a_{n-k}(\ell)
\end{equation}
With the initial condition $a_n(0)=\begin{cases} 0\, , \, if \, n\neq 0 \\ 1 \, , \, if \, n=0    
\end{cases}$. \

Upon simplifying we get the following formulas
\begin{align}
    a_{2n+1}(2)&=\zeta(1-2n) \\
    a_{2n}(2) & =\zeta(-1)\zeta(3-2n)+\zeta(-3)\zeta(5-2n)+\cdots +\zeta(3-2n)\zeta(-1)\\
    a_{2n+1}(3)& =\frac{3}{2}a_{2n}(2)\\
    a_{n}(n) &=\frac{1}{2^n}
\end{align} Some results are tabulated below:

\begin{table}[h]\caption{Residues $H^m(k)$}\vspace{3mm}\label{tab2}
\begin{tabular}{| c |l |} \hline 
$k$  & $H^m(k)$  \\ 
 \hline
$1$ & $ \gamma^m$  \\    \hline
$0$ & $ \frac{m}{2}\gamma^{m-1}$ \\  \hline
$-1$ & $ \gamma^m\left\{\binom{m}{2}\frac{1}{4\gamma^2 }+\binom{m}{1}\frac{\zeta(-1)}{\gamma}\right\}$ \\   \hline
$-2$ & $ \gamma^m \left\{\binom{m}{3}\frac{1}{8 \gamma^3}+\binom{m}{2}\frac{\zeta(-1)}{\gamma^2}\right\}$ \\ \hline
$-3$ & $ \gamma^m \left\{\binom{m}{4}\frac{1}{16 \gamma^4}+\binom{m}{3}\frac{3 \zeta(-1)}{4 \gamma^3}+\binom{m}{2}\frac{\zeta(-1)^2}{\gamma^2}+\binom{m}{1}\frac{\zeta(-3)}{\gamma}\right\}$ \\  \hline
$-4$ & $ \gamma^m \left\{\binom{m}{5}\frac{1}{32 \gamma^5}+\binom{m}{4}\frac{\zeta(-1)}{2\gamma^4}+\binom{m}{3}\frac{3\zeta(-1)^2}{2\gamma^3}+\binom{m}{2}\frac{\zeta(-3)}{\gamma^2}\right\}$ \\ \hline \end{tabular}
\end{table}

\begin{table}[h]
\caption{Values of $a_n$}\vspace{3mm}\label{tab3}
\begin{tabular}{|c|c|c|c|c|} \hline 
    $n$ & $a_n(1)$ & $a_n(2)$ & $a_n(3)$ & $a_n(4)$ \\ \hline
   $1$&  $\frac{1}{2}$ &  0 & 0 & 0 \\ \hline   
   $2$&  $\zeta(-1)$ &  $\frac{1}{4}$ & 0 & 0 \\ \hline
   $3$&  0&  $\zeta(-1)$& $\frac{1}{8}$ & 0\\ \hline
   $4$&  $\zeta(-3)$ & $\zeta(-1)^2$ & $\frac{3}{4}\zeta(-1)$ & $\frac{1}{16}$\\ \hline
   $5$&  0& $\zeta(-3)$ & $\frac{3}{2}\zeta(-1)^2$& $\frac{1}{2}\zeta(-1)$ \\ \hline
   $6$&  $\zeta(-5)$ & $2\zeta(-1)\zeta(-3)$ & $\zeta(-1)^3+\frac{3}{4}\zeta(-3)$ & $\frac{3}{2}\zeta(-1)^2$ \\ \hline
   $7$&  0& $\zeta(-5)$ & $3\zeta(-1)\zeta(-3) $ & $2\zeta(-1)^3+\frac{1}{2}\zeta(-3)$\\ \hline
   $8$&  $\zeta(-7)$ & $2\zeta(-1)\zeta(-5)+\zeta(-3)^2$ & $\frac{3}{4}\zeta(-5)+3\zeta(-1)^2\zeta(-3)$ &$\zeta(-1)^4+3\zeta(-1)\zeta(-3)$ \\ \hline
   $9$&  0& $\zeta(-7)$ & $3\zeta(-1)\zeta(-5)+\frac{3}{2}\zeta(-3)^2$ & $\frac{1}{2}\zeta(-5)+6\zeta(-1)^2\zeta(-3)$\\ \hline 
\end{tabular}
\footnotetext[1]{The upper triangular section are all zeros}
\end{table}

The following formal power series captures $a_n$ well. 
\begin{lemma}\label{formal power series}
    Consider the formal power series \[f(z)=\sum_{n=1}^{\infty} (-1)^n \zeta(1-n)z^n\]
    Then \begin{equation}\label{generating function of a}
        [f(z)]^m =\left(\frac{z}{2}+\zeta(-1)z^2+\zeta(-3)z^4+\zeta(-5)z^6+\cdots\right)^m=\sum_{n\geq m} a_n(m)z^n
    \end{equation}
and \[(\gamma+ f(x))^m=\sum_{n=0}^{\infty} H^m(1-n) x^n\]
\end{lemma}
\begin{proof}
Repeatedly multiplying out the terms using the Cauchy product and considering the reduction formula in (\ref{reduction of an}) gives the lemma. \end{proof}
\subsection{Higher order poles}\par
Let $K^m(n)$ and $L^m(n)$ denote the coefficient of $(s-n)^{-2}$ and $(s-n)^{-3}$ respectively of the Laurent expansion of $\mathcal{H}^m(s)$ at $s=n$. Then the coefficients could be written in the form $$\sum_{k=0}^m \binom{m}{k}\gamma^{-k}b_{n}(k)$$
and $b_n(k)$ could be written as $ka_n(k-1)$ and $k(k-1)a_n(k-2)$ respectively, here $a_n(k)$ is the same as in (\ref{an}). This is done the same in showing the reduction formula of $a_n$.

The following result exhibits a remarkable relation among coefficients.

\begin{theorem}\label{Laurent expansion of H}
    Let $\mathrm{H}^m(n,k)$ denote the $(s-n)^{-k}$ coefficient of the Laurent expansion of $\pHz{m}{s}$ expanded at $s=n$. Then
    \begin{equation}\label{coefficient of expansion}\mathrm{H}^m(1-n,k)=\sum_{\ell=0}^m \binom{m}{\ell}\gamma^{m-\ell} a_n(\ell,k)\end{equation}
    Where $a_n(\ell, k)$ is a triple sequence of rational numbers and takes the form $\ell(\ell-1)\cdots(\ell -k+2)a_n(\ell-k+1)$. Moreover, we have the formal power series \begin{equation}
        \sum_{n=0}^{\infty}\mathrm{H}^m(1-n,k+1)x^n=m(m-1)\cdots(m-k+1)(\gamma+f(x))^{m-k}
    \end{equation}
\end{theorem}
\begin{proof}
    We will use the symbol $\sum_{n=0}$ to denote the sum has an upper bound that is yet determined. Using (\ref{Asymp. Harmonic}) we get\begin{equation}\label{relating m+1 to m}
        \mathcal{H}^{m+1}(s)=-{\mathcal{H}^m}'(s)+\gamma \mathcal{H}^m (s)+\sum_{a=1}(-1)^a \zeta(1-a)\mathcal{H}^m(a+s)+\cdots
    \end{equation} Pulling out the terms gives
    \[\mathrm{H}^{m+1}(n,k)=(k-1)\mathrm{H}^m(n,k-1)+\gamma\mathrm{H}^m(n,k)+\sum_{a=1}(-1)^a \zeta(1-a)\mathrm{H}^m(n,k)\]
If we set $n\mapsto 1-n$ and write H in the form (\ref{coefficient of expansion}), which we know is true for $k=1,2,3$ and is easy to be seen true in general, we get
\[\begin{split}
   & \mathrm{H}^{m+1}(1-n,k)\gamma^{-m}- \mathrm{H}^m(1-n,k)\gamma^{1-m} \\ & =(k-1)\sum_{l=0}^{\infty}\binom{m}{l}\frac{1}{\gamma^l}a_n(l,k-1) +\sum_{a=1} (-1)^a\zeta(1-a)\sum_{l=0}^{\infty}\binom{m}{l}\frac{1}{\gamma^l}a_{n-a}(l,k)
\end{split}\]
summing both sides in $m$ gives a telescoping sum, which is simplified as
\[\mathrm{H}^m(n,k)\gamma^{-m}=\sum_{l=1}^{\infty}\binom{m}{l}\frac{1}{\gamma^{l}}\left\{(k-1)a_n(l-1,k-1)+\sum_{a=1}(-1)^a \zeta(1-a)a_{n-a}(l-1,k)\right\}\]
Equating coefficients gives
\[a_n(l,k)=(k-1)a_n(l-1,k-1)+\sum_{a=1}(-1)^a \zeta(1-a)a_{n-a}(l-1,k)\]
Consider the generating function of both sides. For the sake of induction, let's assume that $a_n(l,k-1)$ could be written in the form $l(l-1)\cdots(l-k+3)a_n(l-k+2)$, which we proved in the prior section for $k=2,3,4$. Set the formal power series $f_{k,l}(x)=a_1(l,k)x +a_2(l,k)x^2+a_3(l,k)x^3+\cdots$. Then we can write the above as 

$$f_{k,l}(x)=(k-1)f_{k-1,l-1}(x)+f(x)f_{k,l-1}(x)$$
I may omit the $x$ and write $f(x)$ as $f$ from now on. Using lemma \ref{formal power series}, we have $f_{k-1,l-1}=(l-1)\cdots(l-k+2)f^{l-k+1}$. Therefore

\[\begin{split}
    f_{k,l}-f\cdot f_{k,l-1}&=(k-1)(l-1)\cdots(l-k+2)f^{l-k+1}\\ 
\frac{f_{k,l}}{f^l}-\frac{f_{k,l-1}}{f^{l-1}}& =(k-1)f^{1-k}\cdot (l-1)\cdots (l-k+2)
\end{split}\]
Summing both sides gives
\[\frac{f_{k,m}}{f^m}-\frac{f_{k,0}}{f^0}=\frac{m(m-1)\cdots(m-k+2)}{f^{k-1}}\]
Note that $f_{k,0}=0$ for $k\geq 2$, Finally we multiply both sides by $f^m$ and equate coefficients to show induction.\end{proof}

A consequence of this is that, if the Euler-Mascheroni constant $\gamma$ is rational, then every coefficient $\text{H}^m(1-n,k)$ is a rational number. This could be a potential tool for proving the irrationality of $\gamma$.

\section{Further studies}First of all, it should be possible to find the constant terms of the Laurent expansion mentioned in theorem \ref{Laurent expansion of H}. For example, using the result \begin{equation}
    \sum_{n=1}^{\infty}\left(H_n^2-\left(\ln(n)+\gamma+\frac{1}{2n}\right)^2\right)=\frac{\ln^2(2\pi)}{2}-\gamma\ln(2\pi)-\frac{\gamma^2}{2}-2\gamma_1-1
\end{equation} given in this 
 \href{https://math.stackexchange.com/questions/2399259/a-closed-form-of-sum-n-1-infty-left-h-n2-left-ln-n-gamma-frac12n-r?rq=1}{math stack exchange post} \cite{2399259}, we can deduce the Laurent series of $\mathcal{H}^2(s)$ expanded at $s=0$ up to the constant term, i.e. \begin{equation}
     \mathcal{H}^2(s)=\frac{1}{s^2}+\frac{\gamma}{s}+\frac{\gamma^2}{2}-1+O(s)
 \end{equation} Secondly, different kinds of Dirichlet series can be studied using theorem \ref{Master theorem of Dirichlet series at negative integers}, as demonstrated. In fact, we have a more general result.
\begin{theorem} \label{Cauchy relation of Dirichlet series at negative integers} Set
$$A_k=\sum_{n\geq 1} a_n n^k x^n\quad ,\quad B_k=\sum_{n\geq 1} b_n n^k x^n\quad ,\quad C_k=\sum_{n\geq 1} c_n n^k x^n $$
Let $a_n=b_0c_n+b_1c_{n-1}+\cdots+b_n c_0$, then $$A_0=-a_0+b_0c_0+c_0B_0+b_0C_0+B_0C_0$$
$$A_k=c_0B_k+b_0C_k+\binom{k}{0}B_0C_k+\binom{k}{1}B_1C_{k-1}+\cdots+\binom{k}{k}B_kC_0$$
\end{theorem}
setting $c_k=1$ gives theorem \ref{Master theorem of Dirichlet series at negative integers}. Theorem \ref{Cauchy relation of Dirichlet series at negative integers} was not mentioned in our study because the author has not found a chance to use this in the paper. Theorem \ref{Cauchy relation of Dirichlet series at negative integers} can be used to find number theoretic results, one of them are formulas regarding $c_m(n)$ in (\ref{cm formula}). For example: 
$$c_{n-1}(1)+c_{n-2}(2)+c_{n-3}(3)+\cdots+c_1(n-1)=\sigma_0(1)+\sigma_0(2)+\sigma_0(3)+\cdots+\sigma_0(n)-n$$
Thirdly, different kinds of harmonic zeta functions can be studied with the aid of the tools used in this paper. For example, the skewed harmonic numbers are defined as $\overline{H}_n=1-\frac{1}{2}+\frac{1}{3}-\cdots+\frac{(-1)^{n+1}}{n}$. Studies concerning the Dirichlet series generated by $\overline{H}_n$ are studied by \cite{alzer2020four}, \cite{boyadzhiev2009alternating} and \cite{flajolet1998euler}. For another example, the negative values of the multiple zeta values can be studied by repeatedly applying the technique in theorem \ref{Master theorem of Dirichlet series at negative integers}. Since the analytical properties are very similar, its Laurent series coefficients may be of similar form as (\ref{coefficient of expansion}).

\bibliographystyle{plain}
\bibliography{bibliography.bib}

\end{document}